\documentclass[a4paper]{amsart}
\usepackage{amssymb}
\usepackage{amsmath}
\usepackage{mathtools}
\usepackage{todonotes}
\usepackage{esint}
\definecolor{cobalt}{RGB}{61,89,171}
\usepackage[colorlinks,citecolor=cobalt,linkcolor=cobalt,urlcolor=cobalt,pdfpagemode=UseNone]{hyperref}

\newtheorem{theorem}{Theorem}[section]
\newtheorem{corollary}[theorem]{Corollary}

\newtheorem{proposition}[theorem]{Proposition}

\theoremstyle{definition}
\newtheorem{definition}[theorem]{Definition}
\newtheorem{example}[theorem]{Example}

\theoremstyle{remark}
\newtheorem{remark}[theorem]{Remark}

\newcommand{\bR}{{\mathbb R}}

\newcommand{\bC}{{\mathbb C}}

\newcommand{\delbar}{\operatorname{\overline{\partial}}}
\newcommand{\del}{\operatorname{\partial}}
\newcommand{\db}{\operatorname{\overline{\partial}}}

\numberwithin{equation}{section}

\begin{document}

\title[Complex symplectic structures: deformations and cohomology]{Complex symplectic structures: deformations and cohomology}

\author[G. Bazzoni]{Giovanni Bazzoni}
\address{Dipartimento di Scienza ed Alta Tecnologia, Universit\`a degli Studi dell'Insubria, Via Valleggio 11, 22100, Como, Italy}
\email{giovanni.bazzoni@uninsubria.it}

\author[M. Freibert]{Marco Freibert}
\address{Fachbereich Mathematik, Universit\"at Hamburg, Bundesstra\ss e 55, D-20146 Hamburg, Germany}
\email{marco.freibert@uni-hamburg.de}

\author[A. Latorre]{Adela Latorre}
\address{Departamento de Matemática Aplicada, Universidad Politécnica de Madrid,
Avda. Juan de Herrera 4, 28040 Madrid, Spain}
\email{adela.latorre@upm.es}

\author[N. Tardini]{Nicoletta Tardini}
\address{Dipartimento di Scienze Matematiche, Fisiche e Informatiche, Unità di Matematica e Informatica, Università degli Studi di Parma, Parco Area delle Scienze 53/A, 43124 Parma, Italy}
\email{nicoletta.tardini@gmail.com}

\date{}


\begin{abstract}
We show that complex symplectic structures need not be preserved under small deformations, and we find sufficient conditions for this to happen. We study various cohomologies of compact complex symplectic manifolds, obtaining some topological obstructions to their existence.
\end{abstract}

\maketitle

\section{Introduction}

As the name eloquently suggests, complex symplectic structures are the complex analogue of symplectic structures. Indeed, a complex symplectic structure on a complex manifold $X$ is a holomorphic 2-form $\sigma$ which is closed and non-degenerate. In particular, the complex dimension of $X$ is even, say $2n$, and the canonical bundle of $X$ is trivialized by $\sigma^n$. Such structures exist canonically on the holomorphic cotangent bundle of every complex manifold and on the coadjoint orbits of every complex Lie group. Every hyperk\"ahler manifold (see \cite{Hitchin1,Huybrechts}) has an underlying complex symplectic structure, and the same holds for every hypersymplectic manifold (see \cite{DaSw,Hitchin2}). Complex symplectic structures appear in the study of (algebraic) integrable systems (see \cite{Boalch}). Quite recently, they have also been studied in the framework of (pseudo-)holomorphic Hamiltonian systems (see \cite{Wagner}); moreover, they have proved helpful in developing a Hamiltonian formalism for field theories (see \cite{BrFa}).

Similarly to what happens in real symplectic geometry, there are no general techniques for determining which (compact or non-compact) complex manifolds admit a complex symplectic structure. Therefore, the construction of examples is of great importance. Quite a few constructions have appeared over the last years, see for instance \cite{bazzoni-freibert-latorre-meinke,BFLT,BGGL}. These constructions take place on nilmanifolds and solvmanifolds, compact homogeneous spaces of nilpotent and solvable Lie groups, and produce non simply connected examples. To the best of our knowledge, the only examples of compact simply connected irreducible non-(hyper)K\"ahler complex symplectic manifolds of complex dimension at least four are due to Guan \cite{Guan1}, who recently proved that these examples are even not formal (see \cite{Guan2}). Interestingly enough, Guan's examples also start from a nilmanifold. We refer the reader to Section \ref{preliminaries} for the relevant preliminaries on complex symplectic structures and on nil/solvmanifolds.

In this paper we focus on two related aspects of compact complex symplectic manifolds:
\begin{itemize}
    \item the stability and closedness of the complex structure under small deformations;
    \item their cohomological property.
\end{itemize}

A compact complex symplectic surface is either a torus, or a K3 surface, or a primary Kodaira surface. It is well-known that, on each surface, all complex structures on such surfaces are deformation equivalent. Non-stability matters must therefore be addressed in complex dimension at least 4. It is probably known to the experts that complex symplectic structures are not stable under small deformations. However, we were not able to track down any explicit example of this behavior in the literature. We present several examples in Section \ref{non-stability}. In \cite{guan}, Guan gives a cohomological criterion ensuring stability, which was later generalized by Anthes and collaborators, see \cite{anthes-cattaneo-rollenske-tomassini}. In Section \ref{stability} we provide other criteria, involving the Bott-Chern cohomology and the $\mathcal{C}^\infty$-fullness of the complex structure, ensuring stability. We refer to Section \ref{stability}  for the definition of $\Delta(X_t,\omega_{\mathbb C})$.
\begin{theorem}[Theorems \ref{thm:stability-1}, \ref{thm:stability-2}]
Let $X=(M,J_0)$ be a compact complex manifold endowed with a complex symplectic form $\omega_{\mathbb C}$. Let $\{X_t\}_{t\in(-\varepsilon,\varepsilon)}$ be a differentiable family of deformations of $X_0=X$, where $\varepsilon>0$. Suppose that one of the following holds:
\begin{itemize}
\item the upper-semi-continuous function $t\mapsto h^{2,0}_{BC}(X_t)$ is constant;
\item for every $t\neq 0$ the manifold $X_t$ is complex-\,$\mathcal C^{\infty}$-full at the second stage and there is $\gamma_t\in\Delta(X_t,\omega_{\mathbb C})$ satisfying 
$\partial_t\bar{\partial}_t\big(\pi_t^{1,0}(\gamma_t)\big)=0$,
\end{itemize}
Then, the compact complex manifold $X_t$ admits a complex symplectic structure for any $t$ close enough to $0$.
\end{theorem}
Moreover, in Proposition \ref{prop:BC-const} we give sufficient conditions for the function $t\mapsto h^{2,0}_{BC}(X_t)$ to be constant.

In Section \ref{non-closed} we show that the property of being complex symplectic is not closed under small deformations (see Theorem \ref{hs-no-cerrada}). In Section \ref{cohomology}, we derive some inequalities on the dimensions of de Rham, Dolbeault, Bott-Chern, and Aeppli cohomology of a compact complex symplectic manifold, culminating in some topological obstructions to their existence. 

\begin{theorem}[Theorem \ref{th:cohomologies}, Corollary \ref{co:bettinumbers}]
Let $M$ be a $4n$-dimensional compact smooth manifold and let $(J,\omega_\mathbb{C})$ be a complex symplectic structure on $M$ and put $X=(M,J)$. For all $0\leq k,m\leq n$ $\omega_\mathbb{C}^k\wedge\bar \omega_\mathbb{C}^m$ defines a non trivial cohomology class in 
\[
H^{2(k+m)}_{dR}(M)\,, \ H^{2k,2m}_{\delbar}(X)\,, \ H^{2k,2m}_{\del}(X)\,, \ H^{2k,2m}_{BC}(X) \ \mathrm{and} \ H^{2k,2m}_{A}(X) 
\]
for all $0\leq k,m\leq n$. In particular, we have the following inequalities
\begin{itemize}
\item $h^{2k,2m}_{\delbar}(X)\geq 1$\,,
\item $h^{2k,2m}_{\del}(X)\geq 1$\,,
\item $h^{2k,2m}_{BC}(X)\geq 1$\,,
\item $h^{2k,2m}_{A}(X)\geq 1$\,,
\end{itemize}
for all $0\leq k,m\leq n$.\\
Moreover, for the Betti numbers we have
\begin{equation*}
b_{2k}(M)\geq k+1,\qquad b_{2l}(M)\geq 2n-l+1
\end{equation*}
for every $k=1,\ldots, n$ and every $l=n+1,\ldots,2n-1$.
\end{theorem}
Finally, we give an example of a compact complex symplectic manifold whose Fr\"olicher spectral sequence does not degenerate at the second page.

\medskip

\noindent {\bf Acknowledgements.} G.~Bazzoni is partially supported by the PRIN 2022 project “Interactions between Geometric Structures and Function Theories” (code 2022MWPMAB) and by the GNSAGA of INdAM. A. Latorre is partially
supported by grant PID2023-148446NB-I00, funded by
MICIU/AEI/10.13039/501100011033 and by ERDF/EU. N. Tardini has been supported by the PRIN 2022 project “Real and Complex Manifolds: Geometry and Holomorphic Dynamics” (code 2022AP8HZ9), by the GNSAGA of INdAM and by University of Parma through the action Bando di Ateneo 2023 per la ricerca.

\section{Preliminaries}\label{preliminaries}



\subsection{Complex symplectic structures} Thanks to the celebrated Newlander-Nirenberg theorem, a complex manifold $X$ of complex dimension $n$ can be equivalently described as a pair $(M,J)$ where $M$ is a smooth $2n$-dimensional manifold and $J$ is an almost complex structure, that is $J\colon \mathfrak X(M)\to \mathfrak X(M)$ such that $J^2=-\mathrm{Id}$, which is integrable, i.e.
\[
[JX,JY]=[X,Y]+J[JX,Y]+J[X,JY], \text{ for every }X,Y\in\mathfrak X(M)\,;
\]
here $\mathfrak X(M)$ denotes the smooth vector fields on $M$. One can equivalently define $J$ on $\Omega^1(M)$, the smooth $1$-forms on $M$, by putting $(J\alpha)(X)\coloneqq\alpha(JX)$, for every $\alpha\in\Omega^1(M)$ and $X\in\mathfrak X(M)$. Then, extending $J$ to the complexified space $\Omega^1(M,\mathbb C)\coloneqq\Omega^1(M)\otimes\mathbb C$, one has
$$\Omega^1(M,\mathbb C)=\Omega_J^{1,0}(M)\oplus\Omega_J^{0,1}(M),$$
with $\Omega_J^{1,0}(M)=\{\omega\in\Omega^1(M,\mathbb C)\mid J\omega=i\,\omega\}$ and 
$\Omega_J^{0,1}(M)=\overline{\Omega_J^{1,0}(M)}$. This induces a bigrading on the space of complexified $k$-forms, for each $1\leq k\leq 2n$,
\begin{equation}\label{bigraduation}
\Omega^{k}(M,\mathbb C)=\bigoplus_{p+q=k}\Omega_J^{p,q}(M).
\end{equation}
When it is clear from the context, we will omit the reference to $J$ in $\Omega_J^{p,q}(M)$ and simply write $\Omega^{p,q}(M)$.

Now a complex symplectic structure on a complex manifold $X$ is a holomorphic 2-form $\sigma$ which is closed and non-degenerate; notice that $X$ has then even complex dimension. The above description of complex manifolds, specialized to the complex symplectic case, yields the following definition.

\begin{definition}
Let $M$ be a smooth manifold. A \emph{complex symplectic structure} on $M$ is a pair $(J,\omega_{\mathbb C})$ consisting of a complex structure $J$ and a closed non-degenerate $2$-form $\omega_{\mathbb C}\in\Omega^2(M,\mathbb C)$ of bidegree $(2,0)$ with respect to the bigrading induced by $J$ on $\Omega^*(M,\mathbb C)$. We will then say that $\omega_{\mathbb C}$ is a \emph{complex symplectic form} on the complex manifold $(M,J)$, and $(M,J,\omega_{\mathbb C})$ is a \emph{complex symplectic manifold}.
\end{definition}

As we remarked above, the existence of a complex symplectic structure $(J,\omega_{\mathbb C})$ on $M$ implies that the complex dimension of $(M,J)$ is even. 

As shown in \cite[Lemma 3.2]{bazzoni-freibert-latorre-meinke} there is a third, equivalent, way to define a complex symplectic structure. On a smooth manifold $M$, a complex symplectic structure is a pair $(J,\omega)$ where
\begin{itemize}
    \item $J$ is a complex structure;
    \item $\omega$ is a {\em real} symplectic form, i.e. $\omega\in\Omega^2(M)$ is closed and non-degenerate;
\end{itemize}
subject to the compatibility condition
\[
\omega(JX,Y)=\omega(X,JY)\,, \ \textrm{for every }X,Y\in\mathfrak X(M)\,.
\]
In fact, the bijection between the pairs $(J,\omega_{\mathbb C})$ and $(J,\omega)$ is given by $(J,\omega_{\mathbb C})\mapsto (J,\mathrm{Re}(\omega_{\mathbb C}))$ with inverse $(J,\omega)\mapsto (J,\omega-i\,\omega(J\cdot\,,\,\cdot))$.

\subsection{Solvmanifolds and nilmanifolds} In this paper we consider nilmanifolds and solvmanifolds. Recall that a \emph{solvmanifold} $\Gamma\backslash G$ is a compact quotient of a connected and simply connected solvable Lie group $G$ and a discrete subgroup $\Gamma\subset G$. When the Lie group $G$ is nilpotent, $\Gamma\backslash G$ is said to be a \emph{nilmanifold}. Solvmanifolds and nilmanifolds are of special interest in geometry, as certain geometric structures on $\Gamma\backslash G$ can be studied on the Lie algebra $\mathfrak g$ of $G$. For instance, a complex structure $J$ on $\Gamma\backslash G$ is called \emph{invariant} if it comes from a complex structure defined on the Lie algebra $\mathfrak g$ of $G$. Similarly, one can define the notion of invariant complex symplectic structure. The following result, proved in~\cite[Proposition 6.1]{bazzoni-freibert-latorre-meinke}, shows that the existence of complex symplectic forms on solvmanifolds with invariant complex structures can sometimes be detected at the Lie algebra level.

\begin{proposition}\label{prop:BFLM}{\rm\cite{bazzoni-freibert-latorre-meinke}}
Let $\Gamma\backslash G$ be a solvmanifold satisfying $H^{\bullet}_{\mathrm{dR}}(\Gamma\backslash G)\cong H^{\bullet}(\mathfrak g^*)$, and let $J$ be an invariant complex structure on $\Gamma\backslash G$. If $(\Gamma\backslash G,J)$ admits a complex symplectic form $\omega$ (not necessarily invariant), then there is an invariant complex symplectic form $\omega'$ on $\Gamma\backslash G$.
\end{proposition}

Here $H^{\bullet}(\mathfrak g^*)$ denotes the Lie algebra cohomology of $\mathfrak{g}$. This useful result, which applies in particular to all nilmanifolds (by Nomizu theorem, \cite{Nomizu}) and to completely solvable\footnote{A solvmanifold $\Gamma\backslash G$ is {\em completely solvable} if $\mathrm{ad}_X$ has real eigenvalues, for each $X\in\mathfrak{g}$.} solvmanifolds (by Hattori theorem, \cite{Hattori}), can be used to exclude the existence of complex symplectic structures along holomorphic deformations.

\subsection{Deformations of complex manifolds} A \emph{holomorphic family} of compact complex manifolds is a proper holomorphic submersion $\pi\colon\mathcal X\to\Delta$ between two complex manifolds $\mathcal X$ and $\Delta$. The fibers $X_{\mathbf t}=\pi^{-1}(\mathbf t)$ are compact complex manifolds of the same dimension. If $\Delta$ is a contractible manifold (for instance, a polydisc), by~\cite{ehresmann} all the fibers have the same underlying smooth manifold $M$. This allows to view the holomorphic family as a collection of complex manifolds 
$\{X_{\mathbf t}\coloneqq(M,J_{\mathbf t})\}_{\mathbf t\in\Delta}$ where each $J_{\mathbf t}$ is a complex structure on $M$. 
In particular, observe that the right-hand side of \eqref{bigraduation} can vary along the deformation, whereas the-left hand side remains the same.

Let $\mathcal P$ be a property related to compact complex manifolds, such as the dimension of a certain cohomology group or the existence of a specific type of geometric structure. Let $\Delta$ be an open disc around the origin in $\mathbb C$. The property $\mathcal P$ is said to be \emph{stable} if for any holomorphic family of compact complex manifolds $(M,J_{\mathbf t})_{\mathbf t\in\Delta}$ such that the manifold $(M,J_{\mathbf t_0})$ satisfies $\mathcal P$, for some $\mathbf t_0\in\Delta$, one has that every $(M,J_{\mathbf t})$ has the property $\mathcal P$, for $\mathbf t\in\Delta$ sufficiently close to $\mathbf t_0$. A trivial example of a stable property is the value of Betti numbers, since the de Rham cohomology groups are a topological invariant of manifolds. In contrast, Hodge numbers can vary along holomorphic deformations, as shown by Nakamura in~\cite{nakamura}, hence their values are not a stable property. A highly non trivial example of a stable property is the existence of a K\"ahler metric, as proved by Kodaira and Spencer in the 1960s~\cite{kodaira-spencer}. 

Finally, a property $\mathcal P$ of compact complex manifolds is said to be \emph{closed} if whenever a holomorphic family of compact complex manifolds $(M,J_{\mathbf t})_{\mathbf t\in\Delta}$ satisfies $\mathcal P$, for $\mathbf t\in\Delta$ sufficiently close to $\mathbf t_0$, it follows that $(M,J_{\mathbf t_0})$ satisfies $\mathcal P$.



%
%

\section{Non-stability result}\label{non-stability}

In this section we present three explicit examples of small deformations of complex symplectic manifolds which are not complex symplectic. The first two examples are a complex symplectic nilmanifold and solvmanifold of which no small deformation is complex symplectic. The third example is more complicated, as the persistence of complex symplectic structures depends on the specific deformation within a general class.

\begin{example}\label{ex:deformation-1-NLA}
Let $(\Gamma\backslash G,J_0)$ be an $8$-dimensional nilmanifold with invariant complex structure defined by the complex structure equations
\begin{equation}\label{ex:initial-CSnilm-1}
 d\varphi^1=d\varphi^2=0, \quad d\varphi^3=\varphi^{1\bar1}, \quad d\varphi^4=\varphi^{12}.
\end{equation}
This means that $\{\varphi^1,\ldots,\varphi^4\}$ is a basis of $\mathfrak{g}^{(1,0)}\subset\mathfrak{g}^*\otimes\bC$ and $d$ denotes the extension of the Chevalley-Eilenberg differential to $\mathfrak{g}\otimes\bC$. Moreover, $\varphi^{1\bar 1}\coloneqq\varphi^1\wedge \varphi^{\bar1}$ and $\varphi^{12}\coloneqq\varphi^1\wedge\varphi^2$; we will use this notation throughout. A direct computation from~\eqref{ex:initial-CSnilm-1} shows that
\[
\omega_{\mathbb C}=A\,\varphi^{12}+B\,\varphi^{13}+C\,\varphi^{14}+D\,\varphi^{24}, \text{ with }A,B,C,D\in\mathbb C, \ BD\neq 0\,,
\]
is a closed $(2,0)$-form on $(\Gamma\backslash G,J_0)$. Moreover, one can check that $\omega_{\mathbb C}$ is non-degenerate, so $(\Gamma\backslash G,J_0,\omega_{\mathbb C})$ is a complex symplectic manifold.

Observe that the $(0,1)$-form $\varphi^{\bar2}$ defines a Dolbeault cohomology class on $(\Gamma\backslash G,J_0)$. Hence, for each $\mathbf{t}\in\mathbb C$ with $|\mathbf{t}|<1$, a complex structure $J_\mathbf{t}$ on $\Gamma\backslash G$ can  be defined using the following basis of $(1,0)$-forms:
\[
\varphi_{\mathbf{t}}^1\coloneqq\varphi^1, \quad 
\varphi_{\mathbf{t}}^2\coloneqq\varphi^2+\mathbf{t}\,\varphi^{\bar2},\quad
\varphi_{\mathbf{t}}^3\coloneqq\varphi^3,\quad 
\varphi_{\mathbf{t}}^4\coloneqq\varphi^4\,.
\]
Then, the complex structure equations of $(\Gamma\backslash G,J_{\mathbf{t}})$ are
\begin{equation}\label{ex:deformed-CSnilm-1}
d\varphi_{\mathbf{t}}^1=d\varphi_{\mathbf{t}}^2=0, \quad 
d\varphi_{\mathbf{t}}^3=\varphi_{\mathbf{t}}^{1\bar1}, \quad
d\varphi_{\mathbf{t}}^4=\frac{1}{1-|{\mathbf{t}}|^2}\,\varphi_{\mathbf{t}}^{12}
-\frac{{\mathbf{t}}}{1-|{\mathbf{t}}|^2}\,\varphi_{\mathbf{t}}^{1\bar2}.
\end{equation}
Any invariant $(2,0)$-form on $(\Gamma\backslash G,J_{\mathbf{t}})$ is then given by
\[
\Omega_{\mathbb C}=\alpha\,\varphi_{\mathbf{t}}^{12}
	+\beta\,\varphi_{\mathbf{t}}^{13}+\gamma\,\varphi_{\mathbf{t}}^{14}
	+\tau\,\varphi_{\mathbf{t}}^{23}+\theta\,\varphi_{\mathbf{t}}^{24}
	+\zeta\,\varphi_{\mathbf{t}}^{34}\,,
\]
where $\alpha,\beta,\gamma,\tau,\theta,\zeta\in\mathbb C$. Imposing closedness we have
\[
0=d\Omega_{\mathbb C}=-\frac{\zeta}{1-|\mathbf t|^2}\,\varphi^{123}
	+\tau\,\varphi^{12\bar1}-\frac{\mathbf t\,\theta}{1-|\mathbf t|^2}\,\varphi^{12\bar2}
	-\frac{\mathbf t\,\zeta}{1-|\mathbf t|^2}\,\varphi^{13\bar2}-\zeta\,\varphi^{14\bar1}\,,
\]
hence we need $\tau=\zeta=0$ and $\mathbf t\,\theta=0$. If $\mathbf t\neq 0$, then $\theta=0$ and $\Omega_{\mathbb C}^2=0$, so there is no invariant complex symplectic form on $(\Gamma\backslash G,J_{\mathbf{t}})$. By Nomizu's Theorem we know that $H^{\bullet}_{dR}(\Gamma\backslash G)\cong H^{\bullet}(\mathfrak g^*)$, so we can apply Proposition~\ref{prop:BFLM} to conclude that the manifold $(\Gamma\backslash G,J_{\mathbf{t}})$ is not complex symplectic for any $\mathbf t\neq 0$.
\end{example}

A similar situation can of course be obtained in the more general class of solvmanifolds, where one can not hope to compute de Rham cohomology using invariant forms, hence Proposition \ref{prop:BFLM} does not apply, and ad-hoc techniques are needed.

\begin{example}
Let $(M\coloneqq N\times\mathbb{T},J_0)$ be the product of the Nakamura manifold $N$ and a $1$-dimensional complex torus $\mathbb{T}$ with 
global co-frame of $(1,0)$-forms $\left\lbrace\varphi^1,\varphi^2,\varphi^3,\varphi^4\right\rbrace$ satisfying the
complex structure equations
\[
\left\lbrace
\begin{array}{lcl}
d\varphi^1 & =& -\varphi^{1\bar3}\\
d\varphi^2 & = & \varphi^{2\bar3}\\
d\varphi^3 & = & 0\\
d\varphi^4 & = & 0
\end{array}
\right.\,.
\]
Notice that this complex structure is abelian. Moreover, one sees easily that the underlying real Lie algebra is not completely solvable, hence we can not use Hattori Theorem to compute de Rham cohomology using invariant forms, and we can not apply Proposition \ref{prop:BFLM} to restrict to invariant complex symplectic structures. For $A,F\in\mathbb{C}$, $AF\neq 0$,
\[
\omega_{\mathbb{C}}=A\varphi^{12}+F\varphi^{34}
\]
is a $d$-closed non-degenerate holomorphic 2-form, hence it defines a complex symplectic structure on $M$. We consider the following small deformation of $J_0$,
\[
\varphi^1_t\coloneqq\varphi^1,\quad
\varphi^2_t\coloneqq\varphi^2,\quad
\varphi^3_t\coloneqq\varphi^3-t\varphi^{\bar 1},\quad
\varphi^4_t\coloneqq\varphi^4\,,
\]
where $t\in\bC$ is a parameter (see \cite[Theorem 3.3]{angella-otal-ugarte-villacampa}). We denote by $M_t$ the real manifold underlying $M$ endowed with the deformed complex structure, $J_t$. The complex structure equations of $M_t$ are
\[
\left\lbrace
\begin{array}{lcl}
d\varphi^1_t & =& -\varphi^{1\bar3}_t\\
d\varphi^2_t & = & -\bar t\varphi^{12}_t+ \varphi^{2\bar3}_t\\
d\varphi^3_t & = & -t\varphi^{3\bar 1}\\
d\varphi^4_t & = & 0
\end{array}
\right.\,.
\]
We proceed to show that $M_t$ does not admit any complex symplectic structure, if $t\neq 0$. Setting $\varphi^j_t=e^{2j-1}+ie^{2j}$, for $j=1,\ldots,4$, and $t=u+iv$, the real structure equations of the real Lie algebra $\mathfrak{g}$ underlying $M_t$ are $de^7=de^8=0$ together with
\[
\left\lbrace
\begin{array}{lcl}
de^1 & = & -e^{15}-e^{26}\\
de^2 & = & e^{16}-e^{25}\\
de^3 & = & u(-e^{13}+e^{24})-v(e^{14}+e^{23})+e^{35}+e^{46}\\
de^4 & = & -u(e^{23}+e^{14})+v(e^{13}-e^{24})-e^{36}+e^{45}\\
de^5 & = & u(e^{15}+e^{26})+v(-e^{16}+e^{25})\\
de^6 & = & u(e^{16}-e^{25})+v(e^{15}+e^{26})\\
\end{array}
\right.\,.
\]
For $j=1,\ldots,8$ we compute $\rho_j=\mathrm{tr}(J_t\mathrm{ad}(e_j))$, obtaining $\rho_1=-4v$, $\rho_2=4u$ and $\rho_j=0$ for $j=3,\ldots,8$. Since $e_1,e_2\in[\mathfrak{g},\mathfrak{g}]$, we see that if $t\neq 0$ the canonical bundle of $M_t$ is not torsion, by \cite[Theorem 5.3]{Andrada-Tolcachier}, hence $M_t$ does not admit any complex symplectic structure unless $t=0$.

\end{example}

As we have shown, the existence of complex symplectic structures is not a stable property. However, it is sometimes possible to find small deformations of complex symplectic manifolds that are still of complex symplectic type. To show this, we consider small deformations of the Iwasawa manifold $X_0=(M,J_0)$. This is a nilmanifold of complex dimension 3, quotient of the complex Heisenberg group by the lattice of Gauss integers, and is defined by the following complex structure equations:
\[
d\omega^1=d\omega^2=0, \qquad d\omega^3=\omega^{12}\,.
\]
The small deformations of the Iwasawa manifold $X_0$ have been studied by Nakamura in \cite{nakamura}. The deformation space is
\[
\Delta=\big\{\mathbf{t}=(t_{11},t_{12},t_{21},t_{22},t_{31},t_{32})\in \mathbb{C}^6 \mid \ |\mathbf{t}|< \varepsilon \big\}
\]
for $\varepsilon>0$ sufficiently small. The deformations are divided in three classes
\begin{description}
    \item[\it Class (i)] $t_{11}=t_{12}=t_{21}=t_{22}=0$;
    \item[\it Class (ii)] $\mathcal{D}(\mathbf{t})=0$ and $(t_{11},t_{12},t_{21},t_{22})\neq (0,0,0,0)$;
    \item[\it Class (iii)] $\mathcal{D}(\mathbf{t})\neq 0$,
\end{description}
where $\mathcal{D}(\mathbf{t})=t_{11}t_{22}-t_{12}t_{21}$. The deformations in Class $(i)$ are parallelizable, while those in Classes $(ii)$ and $(iii)$ are not. Moreover, the Hodge diamonds of the three deformation classes are different (albeit constant within each class), showing that the Hodge numbers are not stable under small deformations, as we pointed out above.

For Classes $(ii)$ and $(iii)$ one can proceed as in the proof of \cite[Theorem 3.1]{angella-tomassini} and write the complex structure equations of $X_\mathbf{t}$ as
\begin{equation}\label{estructura-Iwasawa}
\left\{
\begin{array}{llll}
d\varphi_{\mathbf{t}}^1 \!\!&\!\!=\!\!&\!\! d\varphi_{\mathbf{t}}^2 = 0,\\[4pt]
d\varphi_{\mathbf{t}}^3 \!\!&\!\!=\!\!&\!\! \sigma_{12}\,\varphi_{\mathbf{t}}^{12}
+ \sigma_{1\bar{1}}\,\varphi_{\mathbf{t}}^{1\bar{1}} + \sigma_{1\bar{2}}\,\varphi_{\mathbf{t}}^{1\bar{2}}
+ \sigma_{2\bar{1}}\,\varphi_{\mathbf{t}}^{2\bar{1}} + \sigma_{2\bar{2}}\,\varphi_{\mathbf{t}}^{2\bar{2}},
\end{array}
\right.
\end{equation}
where the coefficients
$\sigma_{12},\sigma_{1\bar{1}},\sigma_{1\bar{2}},\sigma_{2\bar{1}},\sigma_{2\bar{2}} \in \mathbb{C}$ only depend on $\mathbf{t}$,
and are given by
\begin{equation}\label{coef-estructura-Iwasawa}
\begin{array}{rcl}
\sigma_{12} \!\!&\!\!=\!\!&\!\! - \,\gamma - \bar{\alpha}\,|t_{22}|^2 + \frac{1}{\bar{\gamma}}\,(\sigma_{1\bar{1}}\,\bar{\sigma}_{2\bar{2}})\,,\\[6pt]
\sigma_{1\bar{1}} \!\!&\!\!=\!\!&\!\! \bar{\alpha}\,\bar{\gamma}\,\left(t_{21}+\bar{t}_{21}\mathcal{D}(\mathbf{t})\right)\,,\\[6pt]
\sigma_{1\bar{2}} \!\!&\!\!=\!\!&\!\! \bar{\alpha}\,\left(t_{22}+(t_{12}\,\bar{t}_{11}+t_{22}\,\bar{t}_{12})\,\sigma_{1\bar{1}}\right)\,,
\end{array}\qquad
\begin{array}{rcl}
\sigma_{2\bar{1}} \!\!&\!\!=\!\!&\!\! - \,\alpha\,\gamma\,(t_{11}-\bar{t}_{22}\mathcal{D}(\mathbf{t}))\,,\\[6pt]
\sigma_{2\bar{2}} \!\!&\!\!=\!\!&\!\! - \,\alpha\,\gamma\,(t_{12}+\bar{t}_{12}\mathcal{D}(\mathbf{t}))\,,
\end{array}
\end{equation}
with $\alpha$ and $\gamma$ satisfying
$$
\begin{array}{rcl}
\alpha \!\!&\!\!=\!\!&\!\! \displaystyle\frac{1}{1-|t_{22}|^2-t_{21}\,\bar{t}_{12}}\,,\\[8pt]
\gamma \!\!&\!\!=\!\!&\!\! \displaystyle\frac{1}{1-|t_{11}|^2-t_{12}\bar{t}_{21}-\alpha\,
     \left(|t_{11}|^2t_{21}\bar{t}_{12}+|t_{22}|^2t_{12}\bar{t}_{21}+2\mathfrak{Re}(t_{11}t_{22}\bar{t}_{12}\bar{t}_{21}\right))}\,.
\end{array}
$$
\noindent Observe that, compared to \cite{angella-tomassini}, some of the coefficients have been written in a different way, more suitable for our purposes.

\begin{example}\label{example-iwasawa}
Let $Y_{\mathbf t}=X_{\mathbf t}\times\mathbb T$, where $\mathbb T$ is a complex torus and
$\mathbf t\in\Delta$ with $t_{12}=t_{21}=0$. 
For sufficiently small $\mathbf t$, this is a deformation of the complex  manifold $Y_{\mathbf 0}$, which corresponds to the product of the Iwasawa manifold $X_{\mathbf 0}$ by $\mathbb T$.
Let $\varphi^4$ be an invariant $(1,0)$-form on the torus $\mathbb T$.
By \eqref{estructura-Iwasawa} and \eqref{coef-estructura-Iwasawa}, the structure equations
in terms of the basis $\{\varphi^k_{\mathbf{t}}\}_{k=1}^4$, where $\varphi^4_{\mathbf{t}}=\varphi^4$, are:
\begin{equation}\label{deformation-iwasawa}
\begin{cases}
d\varphi_{\mathbf{t}}^1 \ = \ d\varphi_{\mathbf{t}}^2 \ = \ 0,\\
d\varphi_{\mathbf{t}}^3 \ = \ -\frac{1-|t_{11}t_{22}|^2}{(1-|t_{11}|^2)\,(1-|t_{22}|^2)}\,\varphi_{\mathbf{t}}^{12} +
  \frac{t_{22}}{1-|t_{22}|^2}\,\varphi_{\mathbf{t}}^{1\bar2} -
   \frac{t_{11}}{1-|t_{11}|^2}\,\varphi_{\mathbf{t}}^{2\bar1},\\
d\varphi_{\mathbf{t}}^4 \ = \ 0.
\end{cases}
\end{equation}
Thanks to Proposition \ref{prop:BFLM}, the existence of complex symplectic structures on $Y_{\mathbf t}$ can be studied at the Lie algebra level. Let us consider a generic invariant $(2,0)$-form
\[
\omega_\bC=\alpha\,\varphi_{\mathbf{t}}^{12}+\beta\,\varphi_{\mathbf{t}}^{13}+\gamma\,\varphi_{\mathbf{t}}^{14}+\tau\,\varphi_{\mathbf{t}}^{23}+\theta\,\varphi_{\mathbf{t}}^{24}+\zeta\,\varphi_{\mathbf{t}}^{34}\,,
\]
where $\alpha,\beta,\gamma,\tau,\theta,\zeta\in\mathbb C$,
satisfying $d\omega_\bC=0$ and $\omega_\bC\wedge\omega_\bC\neq 0$. A direct calculation from~\eqref{deformation-iwasawa} shows that
\[\begin{array}{lcl}
d\omega_\bC &=& -\frac{\zeta\,(1-|t_{11}t_{22}|^2)}{(1-|t_{11}|^2)\,(1-|t_{22}|^2)}\,\varphi_{\mathbf{t}}^{124}
 +\frac{\beta\,t_{11}}{1-|t_{11}|^2}\,\varphi_{\mathbf{t}}^{12\bar1} 
 +\frac{\tau\,t_{22}}{1-|t_{22}|^2}\,\varphi_{\mathbf{t}}^{12\bar2} \\[4pt]
 && 
 -\frac{\zeta\,t_{22}}{1-|t_{22}|^2}\,\varphi_{\mathbf{t}}^{14\bar2}
 +\frac{\zeta\,t_{11}}{1-|t_{11}|^2}\,\varphi_{\mathbf{t}}^{24\bar1}
 \end{array}
\]
From here, it is clear that $\zeta=0$, as $|t_{11}t_{22}|^2$ is sufficiently small. 
Moreover, the following conditions arise:
$$\left\{\begin{array}{l}
\beta\,t_{11}=0,\\
\tau\,t_{22}=0,\\
\gamma\,\tau-\beta\,\theta\neq 0.
\end{array}\right.$$
Simply note that the last one comes from the non-degeneracy condition $\omega_\bC\wedge\omega_\bC\neq 0$ with $\zeta=0$.
One can then distinguish several cases:
\begin{itemize}
\item For $t_{11}=t_{22}=0$, it suffices to choose $\alpha,\beta,\gamma,\tau,\theta$ with $\gamma\,\tau-\beta\,\theta\neq 0$ to get a complex symplectic structure on $Y_{\mathbf 0}$ (see also \cite{Guan1}).
\item When $t_{11}=0$ but $t_{22}\neq 0$, one can take $\tau=0$ and $\alpha,\beta,\gamma,\theta$ such that $\beta\,\theta\neq 0$ to obtain a complex symplectic structure. 
\item If $t_{11}\neq 0$ and $t_{22}=0$, then $\beta=0$ and $\alpha,\gamma,\tau,\theta$ are free parameters that must satisfy $\gamma\,\tau\neq 0$ in order to provide an appropriate $\omega_\bC$ on the corresponding $Y_{\mathbf t}$.

\item If $t_{11}t_{22}\neq 0$, then one necessarily has $\beta=\tau=0$. However, this 
implies $\Omega\wedge\Omega=0$,
and there are no complex symplectic structures on $Y_{\mathbf t}$, by Proposition  \ref{prop:BFLM}.
\end{itemize}
As a consequence, one can clearly see that if we deform $Y_{\mathbf 0}$ along $t_{11}t_{22}=0$, then  the existence of complex symplectic structures is preserved. However, when this manifold is deformed along $t_{11}t_{22}\neq 0$, complex symplectic structures cease to exist.

\end{example}

\section{Stability results}\label{stability}

The examples of the previous section make us wonder whether it is possible to characterize the deformations of the complex structure of a complex symplectic manifold along which the existence of complex symplectic structures is preserved.

We recall two existing results giving sufficient conditions for {\em all} deformations of a complex symplectic manifold to be again complex symplectic, thus ensuring stability.

\begin{theorem}[Guan, \cite{guan}]\label{thm-Guan}
Let $X$ be a compact complex symplectic manifold such that
\[
H_{\mathrm{dR}}^2(X;\mathbb C)=H_{\db}^{2,0}(X)\oplus H_{\db}^{1,1}(X)\oplus H_{\db}^{0,2}(X)\,.
\]
Then, every small deformation of $X$ admits a complex symplectic structure.
\end{theorem}

The result has been generalized \cite{anthes-cattaneo-rollenske-tomassini}, where stability holds under some hypotheses on the central fiber of a small deformation:
\begin{itemize}
    \item it satisfies the $\partial\overline\partial$-lemma;
    \item or, more in general, the Frölicher spectral sequence degenerates at the first page.
\end{itemize}
\begin{theorem}[Anthes-Cattaneo-Rollenske-Tomassini, \cite{anthes-cattaneo-rollenske-tomassini}]\label{thm-ACRT}
Every small deformation of a compact complex symplectic manifold satisfying one of the hypotheses above admits a complex symplectic structure.
\end{theorem}

These theorems can not be used to explain the behavior observed in Examples \ref{ex:deformation-1-NLA} and \ref{example-iwasawa}. Indeed, By \cite{DGMS} and \cite{hasegawa2}, nilmanifolds do not satisfy the $\partial\overline\partial$-lemma unless they are tori.\\





    
Next, we investigate a different sufficient condition, involving the dimensions of the Bott-Chern cohomology groups, that ensures stability along special curves. For the sake of completeness, recall that the \emph{Bott-Chern cohomology} of a complex manifold $X$ is defined as
\[
H^{\bullet,\bullet}_{BC}(X)\coloneqq\frac{\text{Ker}\,\partial\cap\text{Ker}\,\overline\partial}{\text{Im}\,\partial\overline\partial}\,.	
\]
	
Also, recall from \cite[Section 3]{Schweitzer} that a {\em smooth} deformation of compact complex manifolds is a smooth submersion $\pi\colon\mathcal{X}\to B$ where $\mathcal{X}$ is a smooth manifold, $B\subset\bR^n$ is a ball and every fiber $X_t=\pi^{-1}(t)$ is a compact complex manifold. In particular, the $X_t$'s are all diffeomorphic.
	
\begin{theorem}\label{thm:stability-1}
Let $X$ be a compact complex symplectic manifold and let $\left\lbrace X_t\right\rbrace_{t\in (-\varepsilon,\varepsilon)}$ be a smooth family of deformations of $X = X_0$, where $\varepsilon> 0$. Suppose that the upper-semi-continuous function $t\mapsto h^{2,0}_{BC}(X_t)$ is constant. Then, the compact complex manifold $X_t$ admits a complex symplectic structure for any $t$ close enough to $0$.
\end{theorem}
\begin{proof}
Let $\left\lbrace \omega_t\right\rbrace$ be a family of Hermitian metrics on $X_t$. For each $t\in (-\varepsilon,\varepsilon)$ consider the associated Bott-Chern Laplacian $\Delta_{BC}^t$ and the corresponding Green operator $G_t$. Using Hodge theory for the Bott-Chern cohomology, we can consider the projection $\mathcal{H}^t_{BC}: A^\bullet(X_t,\mathbb{C})\to \text{Ker}\,\Delta_{BC}^t$ from complex valued differential forms onto their Bott-Chern harmonic part. Let $\pi^{2,0}_t:A^\bullet(X_t,\mathbb{C})\to A^{2,0}(X_t)$ be the projection from the space of complex valued differential forms onto the space of $(2,0)$-forms on $X_t$. For any $t\in (-\varepsilon,\varepsilon)$, define
\[
\Pi_t\coloneqq\left(\mathcal{H}^t_{BC}+\partial_t\overline\partial_t(\partial_t\overline\partial_t)^{*_t}{G_t}\right)\circ \pi^{2,0}_t:A^\bullet(X_t,\mathbb{C})\to \text{Ker}\,\partial_t\cap\text{Ker}\overline\partial_t\cap A^{2,0}(X_t)\,.
\]
Note that $\pi_t$ is the projection of the space of complex forms onto the space of $\partial_t$- and $\overline{\partial}_t$-closed complex $(2,0)$-forms, i.e. the space of $d$-closed complex $(2,0)$-forms.
		
Since by assumption the upper semi continuous map $t\mapsto h^{2,0}_{BC}(X_t)$ is constant, by \cite[Theorem 5]{kodaira-spencer} the family $\left\lbrace\Pi_t\right\rbrace_{t\in (-\varepsilon,\varepsilon)}$ is smooth in $t$. Let $(J_0,\omega_0^\mathbb{C})$ be a complex symplectic structure on $X_0$ and set
\[
\omega_t^\mathbb{C}\coloneqq\Pi_t(\omega_0^\mathbb{C})\,.
\]
		
Then, $\left\lbrace\omega_t^\mathbb{C}\right\rbrace_{t\in (-\varepsilon,\varepsilon)}$ is a smooth family in $t$ of $d$-closed $(2,0)$-forms on $X_t$ and for $t$ close to $0$, $\omega_t^\mathbb{C}$ is non-degenerate. Hence, $(J_t,\omega_t^\mathbb{C})$ is a complex symplectic structure on $X_t$.
\end{proof}

The previous result can be used to explain why the existence of a complex symplectic structure is not preserved along the deformation given in Example~\ref{ex:deformation-1-NLA}.


\begin{example}\label{ex:deformation1-BC-non-preserved}
Let $(\Gamma\backslash G,J_0)$ be the $8$-dimensional nilmanifold with invariant complex structure given in Example~\ref{ex:deformation-1-NLA} and defined by the complex structure equations~\eqref{ex:initial-CSnilm-1}. The Bott-Chern cohomology of $(\Gamma\backslash G,J_0)$ can be computed at the Lie algebra level (see~\cite{angella} together with~\cite{rollenske-tomassini-wang}, for instance), so one can use~\eqref{ex:deformation-1-NLA} to get
\[
H_{BC}^{2,0}(\Gamma\backslash G,J_0)=\langle [\varphi^{12}],\,[\varphi^{13}],\,[\varphi^{14}],\,[\varphi^{24}] \rangle\,.
\]
If we now compute the same cohomology group for $(\Gamma\backslash G, J_{\mathbf t})$, with $\mathbf t\neq 0$, using the corresponding structure equations~\eqref{ex:deformed-CSnilm-1}, we obtain
\[
H_{BC}^{2,0}(\Gamma\backslash G,J_{\mathbf t})=\langle [\varphi^{12}],\,[\varphi^{13}],\,[\varphi^{14}]\rangle\,.
\]
In particular, if we assume $\mathbf t=t\in\mathbb R$, one can clearly see that the function $t\mapsto h^{2,0}_{BC}(X_t)$ is not constant along the deformation.
\end{example}

On the other hand, it is possible to provide another deformation of the manifold $(\Gamma\backslash G,J_0)$ in Example~\ref{ex:deformation-1-NLA} where the hypothesis in Theorem~\ref{thm:stability-1} holds, thus preserving the existence of complex symplectic structures.

\begin{example}\label{ex:deformation1-BC-preserved}
Let $(\Gamma\backslash G,J_0)$ be as in Example \ref{ex:deformation-1-NLA}, with structure equations~\eqref{ex:initial-CSnilm-1}. We consider the Dolbeault comohology class $[\varphi^{\bar1}]\in H^{0,1}_{\bar{\partial}}(\Gamma\backslash G,J_0)$ and consider a differentiable family of compact complex manifolds $\big\{X_t\coloneqq(\Gamma\backslash G,J_t)\big\}_{t\in(-1,1)}$ determined by the $(1,0)$-forms
\[
\varphi_{t}^1\coloneqq\varphi^1+t\,\omega^{\bar1}, \quad 
\varphi_{t}^2\coloneqq\varphi^2,\quad
\varphi_{t}^3\coloneqq\varphi^3,\quad 
\varphi_{t}^4\coloneqq\varphi^4\,.
\]
Then, the complex structure equations of $(\Gamma\backslash G,J_t)$ are
\begin{equation}\label{ex:deformed-CSnilm-1-v2}
d\varphi_{t}^1=d\varphi_{t}^2=0, \quad 
d\varphi_{t}^3=\frac{1}{1-t^2}\,\varphi_{t}^{1\bar1}, \quad
d\varphi_{t}^4=\frac{1}{1-t^2}\,\varphi_{t}^{12}
+\frac{t}{1-t^2}\,\varphi_{\mathbf{t}}^{2\bar1},
\end{equation}
and it is easy to see from here that
$$H_{BC}^{2,0}(\Gamma\backslash G,J_t)=\langle [\omega^{12}],\,[\omega^{13}],\,[\omega^{14}+t\,\omega^{23}],\,[\omega^{24}]\rangle.$$
Therefore, the function $t\mapsto h^{2,0}_{BC}(X_t)$ is constant and there is a complex symplectic structure on every $X_t$, for $t\in(-1,1)$. Indeed, one can check that 
\[
\Omega_{\mathbb C}=\alpha\,\varphi_{t}^{12}
	+\beta\,\varphi_{t}^{13}+\gamma\,(\varphi_{t}^{14}
	+t\,\varphi_t^{23})+\theta\,\varphi_t^{24}\,,
\]
with $\alpha,\beta,\gamma,\theta\in\mathbb C$ such that $t\,\gamma^2-\beta\,\theta\neq 0$, is a closed non-degenerate $(2,0)$-form on $X_t$ for every $t\in(-1,1)$.
\end{example}

Next, we describe another sufficient condition ensuring stability of complex symplectic structures.

\medskip

Let $(M,J)$ be a $2n$-dimensional manifold endowed with an almost complex structure. In~\cite{angella-tomassini-2} (see also~\cite{li-zhang}), the authors introduce the following subgroups of the complex de Rham cohomology groups of $M$:
\[
H^{p,q}_J(M)=\left\{ \mathbf a\in H_{\text{dR}}^{p+q}(M;\mathbb C) \mid \exists \ \alpha\in\Omega_J^{p,q}(M) \mid \mathbf a=[\alpha] \right\}\,.    
\]
The almost complex structure $J$ is said to be 
\emph{complex-\,$\mathcal C^{\infty}$-pure at the $k$-th stage} if
\begin{equation}\label{eq:pure}
\bigcap_{p+q=k}\!\!\!H^{p,q}_J(M)=\big\{[0]\big\}\,,
\end{equation}
and it is called \emph{complex-\,$\mathcal C^{\infty}$-full at the $k$-th stage} if
\begin{equation}\label{eq:full}
H_{\text{dR}}^{k}(M;\mathbb C)=\sum_{p+q=k}H^{p,q}_J(M)\,.    
\end{equation}
When \eqref{eq:pure} and \eqref{eq:full} hold simultaneously, namely when 
\[
H_{\text{dR}}^{k}(M;\mathbb C)=\bigoplus_{p+q=k}H^{p,q}_J(M)\,,
\]
then $J$ is called \emph{complex-\,$\mathcal C^{\infty}$-pure-and-full at the $k$-th stage}. 
\medskip
Let $X=(M,J)$ be a compact complex manifold which is complex-\,$\mathcal C^{\infty}$-full at the second stage and fix a closed form $\omega\in\Omega^2(M,\mathbb C)$. Then
\[
[\omega]\in H_{\text{dR}}^2(M;\mathbb C)=H^{2,0}_J(M)+H^{1,1}_J(M)+H^{0,2}_J(M)\,.
\]
Hence, there are cohomology classes $\mathbf a\in H^{2,0}_J(M)$, $\mathbf b\in H^{1,1}_J(M)$,
and $\mathbf c\in H^{0,2}_J(M)$ such that $[\omega]=\mathbf a+\mathbf b+\mathbf c$. Thus, for $p+q=2$, one can find representatives $\alpha^{p,q}\in\Omega_J^{p,q}(M)$ for each of these classes, such that
\[
\omega=\alpha^{2,0}+\alpha^{1,1}+\alpha^{0,2}+d\gamma\,,
\]
where $\gamma\in\Omega^1(M,\mathbb C)$. Following a similar idea as in \cite{latorre-ugarte-proc}, one can then define the space
\[
\Delta(X,\omega)=\left\{ \gamma\in\Omega^1(M,\mathbb C)
	\ \ \Big\vert \ \begin{array}{c}
		d\gamma=\omega-(\alpha^{2,0}+\alpha^{1,1}+\alpha^{0,2}) \\
		\text{for some closed }\alpha^{p,q}\in\Omega_J^{p,q}(M) 
	\end{array}\right\}\,,
\]
which is non-empty, since $X=(M,J)$ is complex-\,$\mathcal C^{\infty}$-full at the second stage. Let us also observe that $\omega$ is fixed, but the representatives $\alpha^{p,q}$ depend on the bigrading induced by $J$. Also, the decomposition $d=\partial+\bar\partial$ of the exterior derivative and the projection 
$\pi^{p,q}\colon\Omega^*(M,\mathbb C)\to\Omega^{p,q}_J(M)$ depend on the bigrading.

\begin{theorem}\label{thm:stability-2}
Let $X=(M,J_0)$ be a compact complex manifold endowed with a complex symplectic form $\omega_{\mathbb C}$. Let $\{X_t\}_{t\in(-\varepsilon,\varepsilon)}$ be a differentiable family of deformations of $X_0=X$, where $\varepsilon>0$. If for every $t\neq 0$ the manifold $X_t$ is complex-\,$\mathcal C^{\infty}$-full at the second stage and there is $\gamma_t\in\Delta(X_t,\omega_{\mathbb C})$ satisfying 
$\partial_t\bar{\partial}_t\big(\pi_t^{1,0}(\gamma_t)\big)=0$,
then $X_t$ is complex symplectic for any sufficiently small $t$.
\end{theorem}
\begin{proof}
For each $t\in(-\varepsilon,\varepsilon)$, we write $X_t=(M,J_t)$.
For the given $\omega_{\mathbb C}\in\Omega^2(M,\mathbb C)$ one can consider how its bigrading varies along the deformation; in general, one has
\begin{equation}\label{expression1}
\omega_{\mathbb C}=( \omega_{\mathbb C} )_t^{2,0}
	+( \omega_{\mathbb C} )_t^{1,1}
	+( \omega_{\mathbb C} )_t^{0,2}
	\in\Omega^{2,0}_{J_t}(M)\oplus \Omega^{1,1}_{J_t}(M)\oplus\Omega^{0,2}_{J_t}(M),
\end{equation}
where $( \omega_{\mathbb C} )_t^{p,q}=\pi^{p,q}_t(\omega_{\mathbb C})\in \Omega^{p,q}_{J_t}(M)$,
for each $p+q=2$. Moreover, $\omega_{\mathbb C}$ is a complex symplectic form on $X_0$, so
$\omega_{\mathbb C}=( \omega_{\mathbb C} )_0^{2,0}$. 
Note that the above decomposition is unique and the
family $\big\{ ( \omega_{\mathbb C} )_t^{2,0} \big\}_t$ is smooth in $t$; however, the forms $( \omega_{\mathbb C} )_t^{2,0}$ may fail to be closed, for $t\neq 0$.

Now, we have that $X_t$ is complex-\,$\mathcal C^{\infty}$-full at the second stage for every $t\neq 0$. Consequently, 
for $p+q=2$, one can find $\alpha_t^{p,q}\in\Omega_{J_t}^{p,q}(M)$ such that $d\alpha_t^{p,q}=0$ and
\begin{equation}\label{expression2}
\omega_{\mathbb C}=\alpha_t^{2,0}+\alpha_t^{1,1}+\alpha_t^{0,2}+d\gamma_t,
\end{equation}
for some $\gamma_t\in\Delta(X_t,\omega_{\mathbb C})$. Indeed, by hypothesis
the element $\gamma_t$ can be chosen to satisfy the condition 
$\partial_t\bar{\partial}_t\big(\pi_t^{1,0}(\gamma_t)\big)=0$. 
If we now compare~\eqref{expression1} and~\eqref{expression2}, one necessarily has that
\[
( \omega_{\mathbb C} )_t^{2,0}= \alpha_t^{2,0}+\partial_t\big(\pi_t^{1,0}(\gamma_t) \big)\,,
\]
and from this we get $d( \omega_{\mathbb C} )_t^{2,0}=0$. Furthermore, as 
the family $\big\{( \omega_{\mathbb C} )_t^{2,0}\big\}_t$ is smooth in 
$t\in(-\varepsilon,\varepsilon)$ and the corresponding element for $t=0$ is
non-degenerate, we have that $( \omega_{\mathbb C} )_t^{2,0}$ is also non-degenerate
for sufficiently small values of $t$. This gives the result.
\end{proof}

\begin{example}
Consider an $8$-dimensional nilmanifold $\Gamma\backslash G$ with invariant complex structure $J_0$ determined by the complex structure equations
\begin{equation*}
 d\varphi^1=d\varphi^2=0, \quad d\varphi^3=\varphi^{12}, \quad d\varphi^4=\varphi^{13}.
\end{equation*}
This manifold admits the following family of complex symplectic structures:
\begin{equation*}
\omega_{\mathbb C}=A\,\varphi^{12}+B\,\varphi^{13}+C\,\varphi^{14}+D\,\varphi^{23}, \text{ with }A,B,C,D\in\mathbb C, \ CD\neq 0.
\end{equation*}

The $(0,1)$-form $\varphi^{\bar1}$ defines a Dolbeault cohomology class on $(\Gamma\backslash G,J_0)$. One can then define a complex structure $J_\mathbf{t}$, for $\mathbf{t}\in\mathbb C$ with $|\mathbf{t}|<1$, on $\Gamma\backslash G$ determined by the $(1,0)$-forms
\begin{equation*}
\varphi_{\mathbf{t}}^1\coloneqq\varphi^1, \quad 
\varphi_{\mathbf{t}}^2\coloneqq\varphi^2,\quad
\varphi_{\mathbf{t}}^3\coloneqq\varphi^3+\mathbf{t}\,\varphi^{\bar1},\quad 
\varphi_{\mathbf{t}}^4\coloneqq\varphi^4\,.
\end{equation*}
The complex structure equations of $(\Gamma\backslash G,J_{\mathbf{t}})$ are
\begin{equation*}
d\varphi_{\mathbf{t}}^1=d\varphi_{\mathbf{t}}^2=0, \quad 
d\varphi_{\mathbf{t}}^3=\varphi_{\mathbf{t}}^{12}, \quad
d\varphi_{\mathbf{t}}^4=\varphi_{\mathbf{t}}^{13}
-{\mathbf{t}}\,\varphi_{\mathbf{t}}^{1\bar1}.
\end{equation*}
From here, one computes the second de Rham cohomology group for $(,J_{\mathbf{t}})$,
\[
H_{\text{dR}}^2(\Gamma\backslash G)=\langle 
[\varphi_{\mathbf{t}}^{14}],\,
[\varphi_{\mathbf{t}}^{23}],\,
[\varphi_{\mathbf{t}}^{1\bar 1}],\,
[\varphi_{\mathbf{t}}^{1\bar 2}], \,
[\varphi_{\mathbf{t}}^{2\bar 1}],\,
[\varphi_{\mathbf{t}}^{2\bar 2}],\,
[\varphi_{\mathbf{t}}^{\bar1\bar 4}],\,
[\varphi_{\mathbf{t}}^{\bar2\bar 3}]
\rangle\,.
\]
From the description above, one can conclude that the complex manifolds $(\Gamma\backslash G,J_{\mathbf{t}})$ are complex-$\mathcal C^\infty$-full at the second stage for every $\mathbf t\in\mathbb C$. Moreover, they satisfy $\partial_t\bar{\partial}_t\big(\Omega^{1,0}_{J_{\mathbf t}}(\Gamma\backslash G)\big)=0$. In particular, the conditions of Theorem~\ref{thm:stability-2} are satisfied and a complex symplectic structure exists for every $\mathbf t\in (-1,1)$.
\end{example}

This example also satisfies the conditions of Theorem \ref{thm:stability-1}, as one can compute
\[
H_{\text{BC}}^{2,0}=\langle 
[\varphi_{\mathbf{t}}^{12}],\,
[\varphi_{\mathbf{t}}^{13}],\,
[\varphi_{\mathbf{t}}^{14}],\,
[\varphi_{\mathbf{t}}^{23}] \rangle\,,
\]
for every $\mathbf t$. The following result explains why this second condition also holds.

\begin{proposition}\label{prop:BC-const}
Let $\{X_t\}_{t\in(-\varepsilon,\varepsilon)}$, with $\varepsilon>0$, be a differentiable family of deformations of a compact complex manifold $X_0$. Suppose that for every $t\neq 0$ the manifold $X_t$ satisfies the following two properties:
\begin{enumerate}
\item[i)] $X_t$ is complex-\,$\mathcal C^\infty$-full at the second stage,
\item[ii)] every $(1,0)$-form on $X_t$ is $\partial_t\bar{\partial}_t$-closed.
\end{enumerate}
Then, the function $t\mapsto h_{BC}^{2,0}(X_t)$ is constant, for $t\in(-\varepsilon',\varepsilon')$, $0<\varepsilon'<\varepsilon$.
\end{proposition}

\begin{proof}
For each $t\in(-\varepsilon,\varepsilon)$, we write $X_t=(M,J_t)$. As (non-zero) $(2,0)$-forms on $X_t$ cannot be $\partial_t\bar{\partial}_t$-exact, one has
\[
H_{BC}^{2,0}(X_t)=\text{Ker}\big\{ d\colon\Omega^{2,0}_{J_t}(M)\to \Omega^3(M,\bC)\big\}\,.
\]
Moreover, recall that the function $t\mapsto h_{BC}^{2,0}(X_t)$ is upper-semi-continuous~\cite[Lemma 3.2]{Schweitzer}, which means that $h_{BC}^{2,0}(X_0)\geq h_{BC}^{2,0}(X_t)$, for every $t\neq 0$. 

Take $\mathbf 0\neq \mathbf a\in H_{BC}^{2,0}(X_0)$. Then, there exists $\alpha\in\Omega^{2,0}_{J_0}(M)\subseteq\Omega^2(M,\bC)$ such that $\mathbf a=[\alpha]$. Even though the $d$-closed form $\alpha$ has type $(2,0)$ with respect to the bigrading induced by $J_0$, in general one has
\begin{equation}\label{alpha-decomp}
\alpha=\alpha^{2,0}_t+\alpha^{1,1}_t+\alpha^{0,2}_t
	\in\Omega^{2,0}_{J_t}(M)\oplus\Omega^{1,1}_{J_t}(M)\oplus\Omega^{0,2}_{J_t}(M),
\end{equation}
where $\alpha^{p,q}_t=\pi_t^{p,q}(\alpha)\in\Omega^{p,q}_{J_t}(M)$, for each $p+q=2$. 

Since the natural map $H_{BC}^{2,0}(X_0)\to H^2_{\text{dR}}(M;\mathbb C)$ is in general not injective, we consider two cases.

Let us first suppose that there exists $\beta\in\Omega^1(M,\bC)$ such that $\alpha=d\beta$. Considering the bidegree of $\beta$ along the deformation, we have
$$\beta=\beta^{1,0}_t+\beta^{0,1}_t\in\Omega^{1,0}_{J_t}(M)\oplus\Omega^{0,1}_{J_t}(M).$$
Consequently, the equality $\alpha=d\beta=(\partial_t+\bar{\partial}_t)\beta$, together with~\eqref{alpha-decomp}, gives $\alpha^{2,0}_t=\partial_t\beta^{1,0}_t$. Therefore, 
\[
d\alpha^{2,0}_t=\bar{\partial}_t\partial_t\beta^{1,0}_t=0\,,
\]
due to assumption ii). Therefore, $[\alpha^{2,0}_t]\in H_{BC}^{2,0}(X_t)$.

We now assume that $\alpha$ is not $d$-exact. Then, $\mathbf a=[\alpha]$ defines a non-zero cohomology class in $H_{\text{dR}}^2(M;\mathbb C)$. Since the manifold $X_t$ is complex-\,$\mathcal C^\infty$-full at the second stage for every $t\neq 0$, it is possible to find $\eta^{p,q}_t\in\Omega^{p,q}_{J_t}(M)$, for $p+q=2$, such that 
\[
d\eta^{p,q}_t=0 \text{ \ and \ } \alpha=\eta^{2,0}_t+\eta^{1,0}_t+\eta^{0,2}_t+d\gamma\,,
\]
where $\gamma=\gamma^{1,0}_t+\gamma^{0,1}_t\in\Omega^{1,0}_{J_t}(M)\oplus\Omega^{0,1}_{J_t}(M)=\Omega^1_{\mathbb C}(M)$. Comparing the previous expression with~\eqref{alpha-decomp}, one gets
$\alpha^{2,0}_t=\eta^{2,0}_t+\partial_t\gamma^{1,0}_t$. This gives $d\alpha^{2,0}_t=0$ as a consequence of ii), thus $[\alpha^{2,0}_t]\in H_{BC}^{2,0}(X_t)$.

We are left with showing that if we consider $[\alpha], [\tilde{\alpha}]\in H^{2,0}_{BC}(X_0)$ such that $[\alpha] \neq [\tilde{\alpha}]$, then $[\alpha^{2,0}_t]\neq [\tilde{\alpha}^{2,0}_t]$ in $H_{BC}^{2,0}(X_t)$ for $t$ sufficiently small. This follows from a continuity argument due to the uniqueness of the decomposition~\eqref{alpha-decomp} and the fact that there are no $\partial_t\bar{\partial}_t$-exact $(2,0)$-forms on $X_t$.

\end{proof}


\begin{remark}
The holomorphic deformations considered in Examples~\ref{ex:deformation-1-NLA} and~\ref{ex:deformation1-BC-preserved} are not complex-$\mathcal C^\infty$-full at the second stage. In the case of Example~\ref{ex:deformation-1-NLA}, one can check from the structure equations \eqref{ex:deformed-CSnilm-1} that the
$2$-form $\varphi^{23}-(1-\left|\mathbf t\right|^2)\varphi^{4\bar1}+{\mathbf t}\,\varphi^{\bar2\bar3}$ defines a de Rham cohomology class that cannot be represented by a form of pure degree $(2,0)$, $(1,1)$ or $(0,2)$. For the deformation in Example~\ref{ex:deformation1-BC-preserved}, one proceeds similarly from~\eqref{ex:deformed-CSnilm-1-v2} considering the $2$-form $\varphi^{23}-\varphi^{4\bar1}$.
\end{remark}

\begin{remark}
A direct computation shows that the deformation considered in Example \ref{example-iwasawa} (for $t_{11}=0$ and $t_{22}\neq 0$) does not satisfy the hypotheses of Theorems \ref{thm-Guan}, \ref{thm:stability-1} and \ref{thm:stability-2} even though the existence of complex symplectic structures is preserved. This shows that such hypothesis are sufficient but not necessary to ensure stability. 
\end{remark}

\section{Complex symplectic structures  not closed under deformations}\label{non-closed}

In this section we prove that the existence of a complex symplectic structure is not a closed property under small deformations. 

\begin{theorem}\label{hs-no-cerrada}
There exists a holomorphic family of compact complex manifolds $\{X_{t}\}_{t\in\Delta}$ of complex
dimension $4$, where $\Delta=\{t\in\mathbb C\mid |t|<1\}$, such that $X_{t}$ admits a complex
symplectic structure for every $t\in\Delta\setminus\{0\}$, but $X_{0}$ is not a complex symplectic manifold.
\end{theorem}

\begin{proof}
Let $X_{0}=(M,J_{0})$ be an $8$-dimensional nilmanifold endowed with an invariant complex structure defined
by the structure equations:
\[
\left\{\begin{array}{l}
d\varphi^1 \ = \ d\varphi^2 \ = \ 0,\\
d\varphi^3 \ = \ \varphi^{12},\\
d\varphi^4 \ = \ i\,\varphi^{1\bar1}+\varphi^{1\bar2}+\varphi^{2\bar1}.
\end{array}\right.
\]
Note that the $(0,1)$-forms $\varphi^{\bar1}$ and $\varphi^{\bar 2}$ define non-zero Dolbeault cohomology classes on~$X_{0}$. Therefore, they provide appropriate directions to perform deformations. For each $t\in\mathbb C$ such that $|t|<1$, define an invariant complex structure $J_{t}$ on $M$ given by the following
basis of $(1,0)$-forms:
\[
\eta_t^1=\varphi^1+t\,\varphi^{\bar1}-i\,t\,\varphi^{\bar2}, \quad \eta_t^2=\varphi^2, \quad 
    \eta_t^3=\varphi^3, \quad \eta_t^4=\varphi^4\,.
\]
The complex structure equations for $X_{t}=(M,J_{t})$ are
\begin{equation}\label{def-hs-no-cerrada}
\left\{\begin{array}{l}
d\eta_t^1 \ = \ d\eta_t^2 \ = \ 0,\\
d\eta_t^3 \ = \ \frac{1}{1-|t|^2}\,\big(\eta_t^{12}+t\,\eta_t^{2\bar1}-i\,t\,\eta_t^{2\bar2}\big),\\
d\eta_t^4 \ = \ \frac{1}{1-|t|^2}\,\big(2\,\bar t\,\eta_t^{12}+i\,\eta_t^{1\bar1}
   +\eta_t^{1\bar2}+\eta_t^{2\bar1}-i\,t\,\eta_t^{2\bar2} \big).
\end{array}\right.
\end{equation}
Let us remark that we recover $X_{0}$, for $t=0$. A generic invariant $(2,0)$-form is
\[
\Omega=\alpha\,\eta_t^{12}+\beta\,\eta_t^{13}+\gamma\,\eta_t^{14}+\tau\,\eta_t^{23}
    +\theta\,\eta_t^{24}+\zeta\,\eta_t^{34}\,,
\]
where $\alpha,\beta,\gamma,\tau,\theta,\zeta\in\mathbb C$. It is non-degenerate if and only if $\alpha \zeta+\gamma\,\tau-\beta\,\theta\neq 0$. We compute
\[
\begin{array}{lcl}
d\Omega &=& -\,\frac{1}{1-|t|^2}\Big( 2\,\bar t\,\zeta\,\eta_t^{123} - \zeta\,\eta_t^{124} +
  (t\,\beta-i\,\theta+\gamma)\,\eta_t^{12\bar1} - i\,(t\,\beta-i\,\theta+t\,\gamma)\,\eta_t^{12\bar2}\\
&& -\,i\,\zeta\eta_t^{13\bar1}-\zeta\,\eta_t^{13\bar2}-\zeta\,\eta_t^{23\bar1}+i\,\zeta\,t\,\eta_t^{23\bar2}
  +t\,\zeta\,\eta_t^{24\bar1}-i\,t\,\zeta\,\eta_t^{24\bar2}\Big)\,.
\end{array}
\]
Imposing $d\Omega=0$, it is straightforward to see that $\zeta=0$ must hold. Therefore, one simply needs to solve the system of equations
\[
\begin{cases}
t\,\beta-i\,\theta+\gamma=0\\
(1-t)\,\gamma=0
\end{cases}
\]
Since $t\neq 1$, it is clear that $\gamma=0$. The problem is then reduced to a single equation
\[
t\,\beta-i\,\theta=0,\quad \text{where}\ \ \beta\,\theta\neq 0\,.
\]
If $t=0$, then $\theta=0$, violating non-degeneracy. As a consequence, there are no complex symplectic structures on $X_{0}=(M,J_{0})$. However, the complex manifold $X_{t}=(M,J_{t})$ with $t\neq 0$  admits complex symplectic structures
\[
\Omega=\alpha\,\eta_t^{12}+\frac it\,\theta\,\eta_t^{13}+\tau\,\eta_t^{23}+\theta\,\eta_t^{24}\,,
\]
where $\alpha,\,\tau,\,\theta\in\mathbb C$ and $\theta\neq 0$.
\end{proof}

\section{Cohomological properties of complex symplectic manifolds}\label{cohomology}

In this section we tackle the cohomological behavior of complex symplectic manifolds.

\subsection{Betti, Hodge and Bott-Chern numbers}

We start by proving the following topological obstruction to the existence of a complex symplectic structure on a compact smooth manifold. Recall that the Aeppli cohomology of a complex manifold $X=(M,J)$ is
\[
H^{\bullet,\bullet}_{A}(X)\coloneqq\frac{\text{Ker}\,\partial\overline\partial}{\text{Im}\,\partial+\text{Im}\,\overline\partial}\,.	
\]

\begin{theorem}\label{th:cohomologies}
Let $M$ be a $4n$-dimensional compact smooth manifold and let $(J,\omega_\mathbb{C})$ be a complex symplectic structure on $M$ and put $X=(M,J)$. For all $0\leq k,m\leq n$ $\omega_\mathbb{C}^k\wedge\bar \omega_\mathbb{C}^m$ defines a non trivial cohomology class in 
\[
H^{2(k+m)}_{dR}(M)\,, \ H^{2k,2m}_{\delbar}(X)\,, \ H^{2k,2m}_{\del}(X)\,, \ H^{2k,2m}_{BC}(X) \ \mathrm{and} \ H^{2k,2m}_{A}(X) 
\]
for all $0\leq k,m\leq n$. In particular,
\begin{itemize}
\item $h^{2k,2m}_{\delbar}(X)\geq 1$\,,
\item $h^{2k,2m}_{\del}(X)\geq 1$\,,
\item $h^{2k,2m}_{BC}(X)\geq 1$\,,
\item $h^{2k,2m}_{A}(X)\geq 1$\,,
\end{itemize}
for all $0\leq k,m\leq n$.
\end{theorem}
\begin{proof}
First of all, note that $\omega_\mathbb{C}^n\wedge\bar\omega_\mathbb{C}^n$ is a volume form on $M$. Hence, its de Rham cohomology class is non-zero. We show now first that also all other cohomology classes of $\omega_\mathbb{C}^n\wedge\bar \omega_\mathbb{C}^n$ are non-zero.
		
Start with $H^{2n,2n}_{\delbar}(X)$ and assume the contrary. Then there exists some $\alpha\in \Omega^{2n,2n-1}(X)$ with $\omega_\mathbb{C}^n\wedge\bar \omega_\mathbb{C}^n=\delbar \alpha$. However, for bidegree reasons, $\del\alpha=0$ and so

\[
d\alpha=\del \alpha+\delbar\alpha=\delbar\alpha=\omega_\mathbb{C}^n\wedge\bar \omega_\mathbb{C}^n\,.
\]
Thus the de Rham cohomology class would be zero as well, a contradiction.
		
A similar argument shows $0\neq [\omega_\mathbb{C}^n\wedge\bar \omega_\mathbb{C}^n]\in H^{2n,2n}_{\del}(X)$. But then there cannot be any $\alpha\in \Omega^{\bullet}(X,\bC)$ with $\del \alpha=\omega_\mathbb{C}^n\wedge\bar \omega_\mathbb{C}^n$. Hence, also $0\neq [\omega_\mathbb{C}^n\wedge\bar \omega_\mathbb{C}^n]\in H^{2n,2n}_{BC}(X)$.
		
Finally, assume that $\omega_\mathbb{C}^n\wedge\bar \omega_\mathbb{C}^n$ defines the trivial cohomology class in Aeppli cohomology. Then $\omega_\mathbb{C}^n\wedge\bar \omega_\mathbb{C}^n=\del \alpha+\delbar \beta$ for $\alpha\in  \Omega^{2n-1,2n}(X)$ and $\beta\in \Omega^{2n,2n-1}(X)$. As above, for bidegree reasons, $\del\alpha=d\alpha$ and $\delbar\beta=d\beta$ and so $\omega_\mathbb{C}^n\wedge\bar \omega_\mathbb{C}^n=d(\alpha+\beta)$, again a contradiction. Hence, also the Aeppli cohomology class of $\omega_\mathbb{C}^n\wedge\bar \omega_\mathbb{C}^n$ is non-zero.
		
These observations imply that for all $0\leq k,m\leq n$ the differential form $\omega_\mathbb{C}^k\wedge\bar \omega_\mathbb{C}^m$ defines non-trivial cohomology classes in all the mentioned cohomologies except the Aeppli cohomology since for all these cohomologies the wedge product on differential forms descends to a product on the cohomology spaces.
		
Finally, assume that for certain $0\leq k,m\leq n$, the Aeppli cohomology class of $\omega_\mathbb{C}^k\wedge\bar \omega_\mathbb{C}^m$ is zero. Then there exists $\alpha\in \Omega^{2k-1,2m}(X)$ and $\beta\in \Omega^{2k,2m-1}(X)$ with $\omega_\mathbb{C}^k\wedge\bar \omega_\mathbb{C}^m=\del \alpha+\delbar \beta$. Since $\omega_{\bC}$ and $\bar\omega_{\bC}$ are both $\del$- and $\delbar$-closed, we then have
\[
\omega_\mathbb{C}^n\wedge\bar \omega_\mathbb{C}^n=\omega_\mathbb{C}^{n-k}\wedge\bar \omega_\mathbb{C}^{n-m}\wedge (\del \alpha+\delbar \beta)=\del(\omega_\mathbb{C}^{n-k}\wedge\bar \omega_\mathbb{C}^{n-m}\wedge \alpha)+\delbar(\omega_\mathbb{C}^{n-k}\wedge\bar \omega_\mathbb{C}^{n-m}\wedge \beta)\,.
\] 
Hence, the Aeppli cohomology class of $\omega_\mathbb{C}^n\wedge\bar \omega_\mathbb{C}^n$ would also be zero, a contradiction.
\end{proof}

By Theorem \ref{th:cohomologies}, for any fixed $l\in \{0,\ldots,2n\}$ the de Rham cohomologie classes $[\omega_{\bC}^k\wedge \bar \omega_{\bC}^m]_{dR}\in H^{2l}_{dR}(X,\bC)$ with $k+m=l$ are non-trivial. In fact, these de Rham cohomology classes are even linearly independent since it is well-known that the identity map induces a well-defined linear map 
\[
H^{2l}_{dR}(X,\bC)\rightarrow \sum_{0\leq k,m\leq n, k+m=l} H^{2k,2m}_A(X)
\]
and the Aeppli cohomology classes $[\omega_{\bC}^k\wedge \bar \omega_{\bC}^m]_{A}$ with $k+m=l$ are non-zero by Theorem \ref{th:cohomologies} and so are linearly independent. Hence, we obtain the following inequalities for the Betti numbers of $X$:
	
\begin{corollary}\label{co:bettinumbers}
Let $M$ be a $4n$-dimensional compact smooth manifold admitting a complex symplectic structure. Then
\begin{equation*}
b_{2k}(M)\geq k+1,\qquad b_{2l}(M)\geq 2n-l+1
\end{equation*}
for every $k=1,\ldots, n$ and every $l=n+1,\ldots,2n-1$.
	\end{corollary}
	
\begin{remark}
Using the above result, we can recover which compact complex surfaces modelled by solvmanifolds admit complex symplectic structures. Let $X$ be a compact complex surface diffeomorphic to a solvmanifold. Then by \cite[Theorem 1]{hasegawa}, $X$ is either a complex torus, or a hyperelliptic surface, or a Inoue surface of type $S^M$, or a primary Kodaira surface, or a secondary Kodaira surface, or a Inoue surface of type $S_\pm$, and, as such, it is endowed with a left-invariant complex structure. It is well-known that the torus and the primary Kodaira surface admit a complex symplectic structure. Moreover, the Inoue surface of type $S^M$, the secondary Kodaira surface and the Inoue surface of type $S_\pm$ have $b_2=0$, hence they do not admit any symplectic structure. Finally, despite having $b_2=2$, the hyperelliptic surface can not admit any complex symplectic structure, since $h^{2,0}_{\delbar}=0$ by \cite{angella-dloussky-tomassini}.
	\end{remark}

\subsection{Fr\"olicher spectral sequence}

The degeneration of the Fr\"olicher spectral sequence seems to be unrelated to the 
existence of complex symplectic structures. For the Iwasawa manifold, 
it is known that $E_1\ncong E_2\cong E_{\infty}$ (see for instance \cite{ceballos-otal-ugarte-villacampa}). Consequently, the manifold $Y_0$ defined in Example~\ref{example-iwasawa} has the same behaviour. We next provide a complex nilmanifold admitting a complex symplectic structure whose Fr\"olicher spectral sequence does not degenerate at the second stage.

\smallskip
Let $X$ be the complex nilmanifold defined by the complex structure equations
\[
d\omega^1=0\,,\quad d\omega^2=\omega^{1\bar1}\,,\quad d\omega^3=\omega^{2\bar1}\,,\quad
d\omega^4=\omega^{3\bar1}\,.
\]
It is easy to check that $\Omega=\omega^{14}-\omega^{23}$ is a complex symplectic structure on $X$. We next compute the spaces $E^{0,2}_r(X)$, which have the following explicit description~\cite[Theorem 1, Corollary 6]{cordero-fernandez-gray-ugarte}:
\[
E^{0,2}_r(X)={\mathcal X}^{0,2}_r(X)/{\mathcal Y}^{0,2}_r(X)\,,
\]
where ${\mathcal Y}^{0,2}_r(X)=\db(\Omega^{0,1}(X))$ for every $r\geq 1$, and
\begin{equation*}
\begin{split}
{\mathcal X}^{0,2}_1(X)&= \{\alpha_{0,2} \in \Omega^{0,2}(X) \mid \db\alpha_{0,2}=0\},\\ 
{\mathcal X}^{0,2}_2(X)&= \{\alpha_{0,2} \in \Omega^{0,2}(X) \mid \db\alpha_{0,2}=\partial\alpha_{0,2}+\db\alpha_{1,1}=0\},\\
{\mathcal X}^{0,2}_3(X) &= \{\alpha_{0,2} \in \Omega^{0,2}(X) \mid \db\alpha_{0,2}=\partial\alpha_{0,2}+\db\alpha_{1,1}=\partial\alpha_{1,1}+\db\alpha_{2,0}=0\},\\
{\mathcal X}^{0,2}_r(X)&= \{\alpha_{0,2} \in \Omega^{0,2}(X) \mid \db\alpha_{0,2}=\partial\alpha_{0,2}+\db\alpha_{1,1}=\partial\alpha_{1,1}+\db\alpha_{2,0}=\partial\alpha_{2,0}=0\}
\end{split}
\end{equation*}
for every $r\geq 4$. Note that these spaces can be equivalently written as
\begin{equation*}
\begin{split}
{\mathcal X}^{0,2}_1(X)&= \{\alpha \in \Omega^{0,2}(X) \, \mid\,  \db\alpha=0 \},\\ 
{\mathcal X}^{0,2}_2(X)&= \{\alpha \in \mathcal{X}^{0,2}_1(X) \, \mid\,  
	\exists\beta\in\Omega^{1,1}(X) \text{ such that } \partial\alpha+\db\beta=0\},\\
{\mathcal X}^{0,2}_3(X) &= \{\alpha \in {\mathcal X}^{0,2}_2(X) \, \mid\,  
	\exists\gamma\in\Omega^{2,0}(X) \text{ such that } \partial\beta+\db\gamma=0\},\\
{\mathcal X}^{0,2}_{r\geq 4}(X)&= \{\alpha \in {\mathcal X}^{0,2}_3(X) \, \mid\,  
	\partial\gamma=0\}.
\end{split}
\end{equation*}
We first observe that ${\mathcal Y}^{0,2}_r(X)=\{0\}$, so $E^{0,2}_r(X)={\mathcal X}^{0,2}_r(X)$ for every $r\geq 1$. Moreover, note that any $(0,2)$-form
$$\alpha=a_{12}\omega^{\bar1\bar2}+a_{13}\omega^{\bar1\bar3}+a_{14}\omega^{\bar1\bar4}
+a_{23}\omega^{\bar2\bar3}-a_{24}\omega^{\bar2\bar4}+a_{34}\omega^{\bar3\bar4},$$
where $a_{i,j}\in\mathbb C$, satisfies $\db\alpha=0$. Consequently, 
$$E^{0,2}_1(X)=\Omega^{0,2}(X).$$ 
Moreover, one can check that 
$$\db\omega^{\bar1\bar2}=\partial \omega^{\bar1\bar2}=0,\qquad
\db\big(\omega^{\bar1\bar4}-\omega^{\bar2\bar3}\big)=
\partial(\omega^{\bar1\bar4}-\omega^{\bar2\bar3}\big)=0,$$
and also
\begin{equation*}
\begin{array}{llllll}
\db \omega^{\bar1\bar3} &=& 0;&&&\\[5pt]
\partial \omega^{\bar1\bar3} &+& \db \omega^{2\bar2} &=& 0;&\\[5pt]
 & & \partial \omega^{2\bar2} &+& \db (-\omega^{13})=  0;\\[5pt]
 & & & & \partial (-\omega^{13})=  0.\\
\end{array}
\end{equation*}
This implies $\omega^{\bar1\bar2}$, $\omega^{\bar1\bar3}$, $\omega^{\bar1\bar4}-\omega^{\bar2\bar3}\in\mathcal{X}^{0,2}_{r}(X)$ for all $r\geq 1$. It then suffices to study
$$\alpha=a_{14}(\omega^{\bar1\bar4}+\omega^{\bar2\bar3})+a_{24}\omega^{\bar2\bar4}
+a_{34}\omega^{\bar3\bar4}, \text{ with }a_{14}, a_{24}, a_{34}\in\mathbb C.$$ 
This element belongs to $\mathcal X^{0,2}_2$ if we are able to find $\beta=\sum_{1\leq r,s\leq 4}b_{rs}\omega^{r\bar s}$ such that
\begin{equation*}
\begin{split}
0=\partial\alpha+\db\beta &=b_{22}\omega^{1\bar1\bar2}+(b_{23}-2\,a_{14})\,\omega^{1\bar1\bar3}
	+(b_{24}-a_{24})\,\omega^{1\bar1\bar4}+a_{24}\,\omega^{1\bar2\bar3}
	-a_{34}\,\omega^{1\bar2\bar4}\\
& +b_{32}\,\omega^{2\bar1\bar2}+b_{33}\omega^{2\bar1\bar3}+b_{34}\omega^{2\bar1\bar4}
	+b_{42}\omega^{3\bar1\bar2}+b_{43}\omega^{3\bar1\bar3}+b_{44}\omega^{3\bar1\bar4}.
\end{split}
\end{equation*}
Note that the existence of such $\beta$ requires $a_{24}=a_{34}=0$. Therefore, 
$$E_2^{0,2}(X)=\langle \omega^{\bar1\bar2},\omega^{\bar1\bar3}, \omega^{\bar1\bar4}-\omega^{\bar2\bar3}, \omega^{\bar1\bar4}+\omega^{\bar2\bar3}\rangle.$$
Moreover for $\alpha=\omega^{\bar1\bar4}+\omega^{\bar2\bar3}$ we have
$$\beta=b_{11}\omega^{1\bar1}+b_{12}\omega^{1\bar2}+b_{13}\omega^{1\bar3}
+b_{14}\omega^{1\bar4}+b_{21}\omega^{2\bar1}+2\omega^{2\bar3}
+b_{31}\omega^{3\bar1}+b_{41}\omega^{4\bar1}.$$
The element $\alpha$ belongs to $\mathcal X^{0,2}_3(X)$ if one can find $\gamma=\sum_{1\leq i<j\leq 4}c_{ij}\omega^{ij}$ satisfying
$$0=\partial\beta+\db\gamma=-c_{13}\omega^{12\bar1}-2\omega^{12\bar2}-(c_{14}+c_{23})\omega^{13\bar1}-c_{24}\omega^{14\bar1}-c_{24}\omega^{23\bar1}-c_{34}\omega^{24\bar1}.$$
It is clear that such $\gamma$ do not exist, so $\alpha\notin\mathcal X^{0,2}_3(X)$. Therefore
$$E_3^{0,2}(X)=\langle \omega^{\bar1\bar2},\omega^{\bar1\bar3}, \omega^{\bar1\bar4}-\omega^{\bar2\bar3}\rangle=E_{\infty}^{0,2}(X).$$





\end{document}